\documentclass[sn-mathphys,Numbered]{sn-jnl}% Math and Physical Sciences Reference Style
\usepackage{graphicx}%
\usepackage{multirow}%
\usepackage{amsmath,amssymb,amsfonts}%
\usepackage{amsthm}%
\usepackage{mathrsfs}%
\usepackage[title]{appendix}%
\usepackage{xcolor}%
\usepackage{textcomp}%
\usepackage{manyfoot}%
\usepackage{booktabs}%
\usepackage{algorithm}%
\usepackage{algorithmicx}%
\usepackage{algpseudocode}%
\usepackage{listings}%
\usepackage{enumitem}

\theoremstyle{thmstyleone}%
\newtheorem{theorem}{Theorem}%  meant for continuous numbers
\newtheorem{lemma}{Lemma} 
\newtheorem{corollary}{Corollary}

\theoremstyle{thmstyletwo}%
\newtheorem{example}{Example}%
\newtheorem{remark}{Remark}%

\theoremstyle{thmstylethree}%
\newtheorem{definition}{Definition}%

\begin{document}

\title[Article Title]{Circular External Difference Families: Construction and Non-Existence}


\author*[1]{\fnm{Huawei} \sur{Wu}}\email{wuhuawei1996@gmail.com}

\author[2]{\fnm{Jing} \sur{Yang}}

\author[3]{\fnm{Keqin} \sur{Feng}}
\equalcont{This research of K. Feng is supported by the National Natural Science Foundation of China under Grant 12031011.}

\affil[1,2,3]{\orgdiv{Department of Mathematical Sciences}, \orgname{Tsinghua University}, \orgaddress{\city{Beijing}, \postcode{100084}, \country{China}}}


\abstract{The circular external difference family and its strong version, which themselves are of independent combinatorial interest, were proposed as variants of the difference family to construct new unconditionally secure non-malleable threshold schemes. In this paper, we present new results regarding the construction and non-existence of (strong) circular external difference families, thereby solving several open problems on this topic.}




\maketitle

\section{Introduction}
Throughout this paper, $\mathbb{C}$, $\mathbb{Q}$, $\mathbb{Z}$, $\mathbb{N}_+$ and $\mathbb{N}$ denote the set of real numbers, rational numbers, integers, positive integers and non-negative integers, respectively. If $m$ is a positive integer, we use $[m]$ to denote the set $\{1,2,\cdots,m\}$ and use $\mathbb{Z}_m$ to denote the ring $\mathbb{Z}/m\mathbb{Z}$. If $q$ is a prime power, we use $\mathbb{F}_q$ to denote the finite field with $q$ elements.

Let $(G,\cdot)$ be a finite group of order $n$. For any non-empty subsets $A$, $B$ of $G$, we define $\Delta(A,B)$ to be the following multiset
$$\{\{ab^{-1}:\ a\in A,\ b\in B\}\}.$$
A non-empty subset $D$ of $G$ is called an $(n,l;\lambda)$-difference set (DS) in $G$ if $|D|=l$ and each element $g\in G\backslash\{1_G\}$ occurs exactly $\lambda$ times in $\Delta(D,D)$. Obviously, a necessary condition for the existence of an $(n,l;\lambda)$-difference set is $l(l-1)=\lambda(n-1)$.

A convenient representation for a multiset with elements in $G$ is through the use of the group ring
$$\mathbb{Z}[G]=\{\sum\limits_{g\in G}a_gg:\ a_g\in\mathbb{Z}\}.$$
Indeed, any multiset $A$ with elements in $G$ can be associated with the element $\sum_{g\in G}n_gg$ in $\mathbb{Z}[G]$, where $n_g$ is the number of occurrences of $g$ in $A$; conversely, any element of $\mathbb{Z}[G]$ with coefficients in $\mathbb{N}$ corresponds to a multiset with elements in $G$. Therefore, in the group ring notation, a non-empty subset $D$ of $G$ is an $(n,l;\lambda)$-difference set if and only if
\begin{equation}\label{20231010equation1}
  DD^{(-1)}=l+\lambda(G-1_{G})\quad\mbox{in }\mathbb{Z}[G],
\end{equation}
where $D^{(-1)}=\{d^{-1}:\ d\in D\}$.

If $G$ is an abelian group, then we can characterize difference sets in $G$ using the characters of $G$. Let $\widehat{G}$ denote the character group of $G$. For any $\chi\in\widehat{G}$, we can extend $\chi$ linearly to a multiplicative function on $\mathbb{Z}[G]$ via
$$\chi(\sum\limits_{g\in G}a_gg)=\sum\limits_{g\in G}a_g\chi(g)\in\mathbb{C}.$$
As with the group $G$, the elements of the group ring $\mathbb{Z}[G]$ can be distinguished by the characters of $G$; that is, for any $x$, $y\in\mathbb{Z}[G]$, $x=y$ if and only if $\chi(x)=\chi(y)$ for all $\chi\in\widehat{G}$. In the character notation, the condition (\ref{20231010equation1}) is equivalent to the following equality:
$$|\chi(D)|^2=\chi(D)\overline{\chi(D)}=\begin{cases}
    l^2,       & \chi=1_{\widehat{G}},                            \\
    l-\lambda, & \chi\in\widehat{G}\backslash\{1_{\widehat{G}}\},
  \end{cases}$$
where $1_{\widehat{G}}$ is the trivial character that maps all element of $G$ to $1\in\mathbb{C}$.

The difference set is a fundamental and significant mathematical structure in combinatorial design theory. For basic results on this topic, we refer to \cite{storer1967cyclotomy}. Over the past few decades, numerous generalizations and variants of the difference set have been proposed and extensively studied in combinatorics. One such concept that extends the idea of the difference set to multiple subsets is known as the difference family.

\begin{definition}\label{20231110def1}
  A family $\{A_1,\cdots,A_m\}$ of $m$ pairwise disjoint non-empty subsets of $G$ is called an $(n,m,l;\lambda)$-difference family (DF) in $G$ if
  \begin{enumerate}[label=(\arabic*)]
    \item $|A_i|=l$ for any $1\le i\le m$;
    \item each element $g\in G\backslash\{1_G\}$ occurs exactly $\lambda$ times in $\cup_{i=1}^m\Delta(A_i,A_i)$.
  \end{enumerate}
\end{definition}

As in the case of the difference set, a family $\{A_1,\cdots,A_m\}$ of $m$ pairwise disjoint non-empty subset of $G$ is an $(n,m,l;\lambda)$-difference family if and only if $|A_i|=l$ for any $1\le i\le m$ and
$$\sum\limits_{j=1}^mA_jA_j^{(-1)}=ml\cdot 1_G+\lambda(G-1_G)\quad\mbox{in }\mathbb{Z}[G].$$

Some generalizations and variants of the difference family were proposed primarily due to their applications in cryptography and information security. In 1971, V. L. Levenshtein introduced in \cite{levenshtein1971one} a significant variant called the external difference family, which was found to be useful in constructing optimal comma-free codes for secure synchronous communication. Later, Ogata el al. discovered in \cite{ogata2004new} that external difference families and strong external difference families could be employed to construct secure authentication codes and secret sharing schemes. The external difference family and its strong version are defined as follows.

\begin{definition}
  A family $\{A_1,\cdots,A_m\}$ of $m$ pairwise disjoint subsets of $G$ is called an $(n,m,l;\lambda)$-external difference family (EDF) in $G$ if
  \begin{enumerate}[label=(\arabic*)]
    \item $|A_i|=l$ for any $1\le i\le m$;
    \item each element $g\in G\backslash\{1_G\}$ occurs exactly $\lambda$ times in $\cup_{1\le i\ne j\le m}\Delta(A_i,A_j)$, or equivalently,
          $$\sum\limits_{1\le i\ne j\le m}A_iA_j^{(-1)}=\lambda(G-1_G)\quad\mbox{in }\mathbb{Z}[G].$$
  \end{enumerate}
  The family $\{A_1,\cdots,A_m\}$ is called an $(n,m,l;\lambda)$-strong external difference family (SEDF) in $G$ if
  \begin{enumerate}[label=(\arabic*)]
    \item $|A_i|=l$ for any $1\le i\le m$;
    \item for any $1\le i\le m$, each element $g\in G\backslash\{1_G\}$ occurs exactly $\lambda$ times in $\cup_{1\le j\le m,j\ne i}\Delta(A_i,A_j)$, or equivalently,
          $$\sum\limits_{\substack{1\le j\le m\\j\ne i}}A_iA_j^{(-1)}=\lambda(G-1_G)\quad\mbox{in }\mathbb{Z}[G].$$
  \end{enumerate}
\end{definition}

In contrast to external difference families, the difference families defined in Definition \ref{20231110def1} can be regarded as "internal" difference families. Extensive research has been dedicated to the construction and non-existence of EDFs and SEDFs. We refer to \cite{huczynska2023internal} for an comprehensive survey on this topic.

Recently, in order to construct unconditionally secure non-malleable threshold schemes, S. Veitch and D. R. Stinson proposed in \cite{veitch2023unconditionally} two new variants of the difference family, namely the circular external difference family (CEDF) and the strong circular external difference family (SCEDF), which themselves are of independent combinatorial interest. They presented several basic results on the construction and non-existence of CEDFs and SCEDFs, and raised some open problems for further investigation.

In this paper, we provide an in-depth and comprehensive study of CEDFs and SCEDFs, thereby solving several open problems raised in \cite{veitch2023unconditionally}. Specifically, the remainder of this paper is arranged as follows.
\begin{itemize}
  \item In Section 2, we generalize the cyclotomic construction of CEDFs proposed in \cite{veitch2023unconditionally} from multiple perspectives, which, in particular, enables us to guarantee the existence of a $(q,m,2;1)$-$c$-CEDF in the additive group of the finite field $\mathbb{F}_q$ for some $c\in [m-1]$, provided that $q=4m+1\ge 13$. Additionally, we present a lifting theorem on CEDFs.
  \item In Section 3, we prove that all SCEDFs are trivial, meaning that they can be constructed from $(n,2,l;\lambda$)-SEDFs. We also present new results on the non-existence of such SEDFs with $n=2p$, where $p$ is a prime number.
  \item Finally, Section 4 serves as a brief summary and conclusion.
\end{itemize}

\section{Circular External Difference Families}

We start by recalling the definition of the circular external difference family.

\begin{definition}[{\cite[Definition 4.2]{veitch2023unconditionally}}]
  Let $(G,\cdot)$ be a finite abelian group of order $n$, let $l$, $m\ge 2$ and let $S$ be a non-empty subset of $[m-1]$. A family $\{A_0,\cdots,A_{m-1}\}$ of $m$ pairwise disjoint subsets of $G$ is called an $(n,m,l;\lambda)$-$S$-circular external difference family (CEDF) in $G$ if
  \begin{enumerate}[label=(\arabic*)]
    \item $|A_i|=l$ for any $0\le i\le m-1$;
    \item each element $g\in G\backslash\{1_G\}$ occurs exactly $\lambda$ times in $\cup_{c\in S}\cup_{i=0}^{m-1}\Delta(A_{i+c},A_i)$, or equivalently,
          $$\sum\limits_{c\in S}\sum\limits_{i=0}^{m-1}A_{i+c}A_i^{(-1)}=\lambda(G-1_G),$$
          where the addition of the subscripts is performed in the sense of modulo $m$.
  \end{enumerate}
  Obviously, a necessary condition for the existence of an $(n,m,l;\lambda)$-$S$-CEDF is $|S|\cdot ml^2=\lambda(n-1)$. If $S=\{c\}$ for some $c\in[m-1]$, then an $(n,m,l;\lambda)$-$S$-CEDF is also called an $(n,m,l;\lambda)$-$c$-CEDF.
\end{definition}

The most common approach to constructing internal and external difference families involves utilizing cyclotomy in finite fields, and the same is true for CEDFs and SCEDFs. So, let us first review the relevant definitions.

\begin{definition}
  Let $q$ be a prime power and assume that $q-1=ef$ for two positive integers $e$, $f\ge 2$. Fix a primitive element $\theta$ of $\mathbb{F}_q$. For any $i\in\mathbb{Z}$, define $C_i=\theta^iC$, where $C=\left\langle\theta^e\right\rangle$ denotes the subgroup of $\mathbb{F}_q^*$ generated by $\theta^e$. Then $C_{i_1}=C_{i_2}$ if and only if $i_1\equiv i_2\ (\mbox{mod}\ e)$ and the sets $C_i$ $(0\le i\le e-1)$ are all the cosets of $C$ in $\mathbb{F}_q^*$, which will be called the cyclotomic classes of $\mathbb{F}_q$ of order $e$ (with respect to $\theta$). The cyclotomic numbers of order $e$ are defined by
  $$(i,j)_{e}^{(q)}=\#\{x\in C_i:\ x+1\in C_j\}$$
  for any $i$, $j\in\mathbb{Z}$.
\end{definition}

The following theorem regarding the cyclotomic numbers of order $4$ is well-known and we will employ it in the proof of Corollary \ref{20231012corol1}.

\begin{lemma}[\cite{storer1967cyclotomy}]\label{20231012lemma1}
  Let $p$ be an odd prime number, let $q=p^n$ with $n\in\mathbb{N}_+$ such that $q\equiv 1\ ({\rm{mod}}\ 4)$, and assume that $q-1=ef$ for two positive integers $e$, $f\ge 2$.\\
  \indent If $p\equiv 1\ ({\rm{mod}}\ 4)$, then there exist integers $s$, $t$ such that $q=s^2+4t^2$, $p\nmid st$ and $s\equiv 1\ ({\rm{mod}}\ 4)$. Here $s$ is uniquely determined by these conditions, while $t$ is uniquely determined to differ by at most a sign. If $p\equiv 3\ ({\rm{mod}}\ 4)$ $($and thus $n$ is even$)$, then we take $t=0$ and take $s$ to be the one congruent to $1$ modulo $4$ in $\{\pm p^{n/2}\}$. The $s$ and $t$ mentioned in the rest of this theorem retain their meanings in this paragraph.
  \begin{enumerate}[label=(\arabic*)]
    \item If $q\equiv 1\ ({\rm{mod}}\ 8)$, then
          \begin{align*}
            (1,0)_4^{(q)}=\frac{1}{16}(q-3+2s+8t), \qquad(3,0)_4^{(q)}=\frac{1}{16}(q-3+2s-8t);
          \end{align*}
    \item If $q\equiv 5\ ({\rm{mod}}\ 8)$, then
          \begin{align*}
            (1,0)_4^{(q)}=(3,0)_4^{(q)}=\frac{1}{16}(q-3-2s).
          \end{align*}
  \end{enumerate}
  Note that we omit the other cyclotomic numbers of order $4$ since they are not used in this paper.
\end{lemma}

The following cyclotomic construction of CEDFs in the additive group of any finite field was proposed in \cite{veitch2023unconditionally}.

\begin{theorem}[{\cite[Thereom 4.3]{veitch2023unconditionally}}]\label{20231110them1}
  Let $q$ be a prime power such that $q-1=ml^2$ with $l$, $m\ge 2$,  $\theta$ be a primitive element in $\mathbb{F}_q$,  $\beta=\theta^l$,  $C=\left\langle\beta^m\right\rangle=\left\langle\theta^{ml}\right\rangle$ and  $C_i=\theta^iC$ $(0\le i\le ml-1)$ be the cyclotomic classes of $\mathbb{F}_q$ of order $ml$ $($with respect to $\theta)$.\\
  \indent For any $0\le j\le m-1$, put $A_j=\beta^jC(=C_{jl})$. Then the family $\{A_0,\cdots,A_{m-1}\}$ is a $(q,m,l;1)$-$1$-CEDF in $(\mathbb{F}_q,+)$ if and only if $\{\beta^{1+km}-1:\ 0\le k\le l-1\}$ is a complete set of coset representatives of $H=\left\langle\beta\right\rangle$ in $\mathbb{F}_q^*$. If $l=2$ $($and thus $q=4m+1$, $\beta=\theta^2)$, then the family $\{A_0,\cdots,A_{m-1}\}$ is a $(q,m,2;1)$-$1$-CEDF in $(\mathbb{F}_q,+)$ if and only if $\theta^4-1$ is a non-square element of $\mathbb{F}_q^*$.
\end{theorem}

\indent In Remark 4.1 of \cite{veitch2023unconditionally}, the authors mentioned that for $q>7.867\times 10^8$, there exists a primitive element $\theta$ of $\mathbb{F}_q$ such that $\theta^4-1$ is non-square. Later, in \cite[Theorem 2.28]{paterson2023circular}, the authors verified computationally that this result also holds for all primes $13\le q<10^9$ with $q\equiv1\ ({\rm{mod}}\ 4)$. Similar to them, we utilized the renowned software package SageMath to examine all the prime powers within this range and discovered that $25$ is the only exception.

On the other hand, if we allow $c$ to vary instead of fixing it to $1$, then it is much easier to prove that there exists a $(q,m,2,1)$-$c$-CEDF in $(\mathbb{F}_q,+)$ for some $c\in [m-1]$, provided that $q=4m+1\ge 13$. Before that, we first generalize Theorem \ref{20231110them1} to the case of general $c\in [m-1]$.

\begin{theorem}\label{20231110them2}
  Keep the notations in Theorem \ref{20231110them1}. Then for any $c\in [m-1]$, the family $\{A_0,\cdots,A_{m-1}\}$ is a $(q,m,l;1)$-$c$-CEDF in $(\mathbb{F}_q,+)$ if and only if
  $$\{x-1:\ x\in A_c\}=\{\beta^{c+km}-1:\ 0\le k\le l-1\}$$
  is a complete set of coset representatives of $H=\left\langle\beta\right\rangle$ in $\mathbb{F}_q^*$.\\
  \indent If $l=2$ $($and thus $q=4m+1$, $\beta=\theta^2)$, then for any $c\in [m-1]$, the family $\{A_0,\cdots,A_{m-1}\}$ is a $(q,m,2;1)$-$c$-CEDF in $(\mathbb{F}_q,+)$ if and only if $\theta^{4c}-1$ is a non-square element of $\mathbb{F}_q^*$.
\end{theorem}
\begin{proof}
  The proof is similar to that of Theorem \ref{20231110them1} given in \cite{veitch2023unconditionally}. Noticing that the order of $\beta$ $($resp., $\beta^m)$ in $\mathbb{F}_q^*$ is $ml$ $($resp., $l)$, and any integer $s$ in $\{0,\cdots,ml-1\}$ can be uniquely written as $s=j+mt$ with $0\le t\le l-1$ and $0\le j\le m-1$, we have
  \begin{align*}
    \displaystyle\bigcup_{j=0}^{m-1}\Delta(A_{j+c},A_j) & =\displaystyle\bigcup_{j=0}^{m-1}\Delta(\beta^{j+c}C,\beta^jC)                                                                               \\
                                                        & =\displaystyle\bigcup_{j=0}^{m-1}\{\{\beta^{j+c}\cdot\beta^{mn_1}-\beta^{j}\cdot\beta^{mn_2}:\ 0\le n_1,n_2\le l-1\}\}                       \\
                                                        & =\displaystyle\bigcup_{j=0}^{m-1}\{\{\beta^{j}\cdot(\beta^m)^{n_2}\cdot\big(\beta^c\cdot(\beta^m)^{n_1-n_2}-1\big):\ 0\le n_1,n_2\le l-1\}\} \\
                                                        & =\displaystyle\bigcup_{j=0}^{m-1}\{\{\beta^{j}\cdot(\beta^m)^{t}\cdot\big(\beta^c\cdot(\beta^m)^{k}-1\big):\ 0\le k,t\le l-1\}\}             \\
                                                        & =\{\{\beta^{j+mt}(\beta^{c+mk}-1):\ 0\le k,t\le l-1,\ 0\le j\le m-1\}\}                                                                      \\
                                                        & =\displaystyle\bigcup_{k=0}^{l-1}(\beta^{c+mk}-1)\cdot\{\beta^s:\ 0\le s\le ml-1 \}                                                          \\
                                                        & = \displaystyle\bigcup_{k=0}^{l-1}(\beta^{c+mk}-1)H.
  \end{align*}
  Hence the family $\{A_0,\cdots,A_{m-1}\}$ is a $(q,m,l;1)$-$c$-CEDF in $(\mathbb{F}_q,+)$, i.e.,
  $$\sum\limits_{j=0}^{m-1}A_{j+c}A_j^{(-1)}=\mathbb{F}_q^*,$$
  if and only if
  $\{\beta^{c+km}:\ 0\le k\le l-1\}$ is a complete set of coset representatives of $H=\left\langle\beta\right\rangle$ in $\mathbb{F}_q^*$.\\
  \indent In particular, if $l=2$, then the family $\{A_0,\cdots,A_{m-1}\}$ is a $(q,m,2;1)$-$c$-CEDF in $(\mathbb{F}_q,+)$ if and only if
  $\{\beta^{c+km}:\ 0\le k\le 1\}=\{\beta^{c}-1,\beta^{c+m}-1\}$ is a complete set of coset representatives of $H=\left\langle\beta\right\rangle=\left\langle\theta^2\right\rangle$ in $\mathbb{F}_q^*$. The latter is equivalent to saying that one of $\beta^c-1$ and $\beta^{c+m}-1$ is a square element of $\mathbb{F}_q^*$ and the other is not, which is further equivalent to saying that their product is a non-square element of $\mathbb{F}_q^*$. Since the order of $\beta$ in $\mathbb{F}_q^*$ is $2m$, we have $\beta^m=-1$ and thus $(\beta^c-1)(\beta^{c+m}-1)=(\beta^c-1)(-\beta^c-1)=1-\beta^{2c}=1-\theta^{4c}$. Since $q\equiv 1\ ({\rm{mod}}\ 4)$, $-1$ is a square element of $\mathbb{F}_q^*$, which implies that $1-\theta^{4c}$ is a non-square element of $\mathbb{F}_q^*$ if and only if $\theta^{4c}-1$ is. This proves the second assertion of this theorem.
\end{proof}

\begin{corollary}\label{20231012corol1}
  Let $q$ be a prime power such that $q=4m+1\ge 13$ with $m\ge 2$, and let $\theta$, $\beta$, $C_i$ $(0\le i\le 2m-1)$ and $A_j$ $(0\le j\le m-1)$ maintain their meanings in Theorem \ref{20231110them1}. Then the family $\{A_0,\cdots,A_{m-1}\}$ is a $(q,m,2;1)$-$c$-CEDF in $(\mathbb{F}_q,+)$ for some $c\in [m-1]$.
\end{corollary}

\begin{proof}
  By Theorem \ref{20231110them2}, it suffices to show that there exists $c\in[m-1]$ such that $\theta^{4c}-1$ is a non-square element of $\mathbb{F}_q^*$. Let $\alpha=\theta^{4}$ and let $E_i=\theta^i\left\langle\alpha\right\rangle$ $(0\le i\le 3)$ be the cyclotomic classes of $\mathbb{F}_q$ of order $4$ (with respect to $\theta$). Then
  \begin{align*}
    \#\{c\in [m-1]:\ \theta^{4c}-1\ \mbox{is non-square}\} & =\#\{c\in [m-1]:\ \alpha^{c}-1\in E_1\cup E_3\} \\
                                                           & =\#\{x\in E_0:\ x-1\in E_1\cup E_3\}            \\
                                                           & =(1,0)_4^{(q)}+(3,0)_4^{(q)}.
  \end{align*}
  \indent Assume that $q=p^n$ where $p$ is an odd prime and $n\in\mathbb{N}_+$. If $q\equiv 1\ ({\rm{mod}}\ 8)$, then by Lemma \ref{20231012lemma1}, there exist  integers $s$ and $t$ such that
  \begin{align*}
    (1,0)_4^{(q)}+(3,0)_4^{(q)} & =\frac{1}{16}(q-3+2s+8t)+\frac{1}{16}(q-3+2s-8t) \\
                                & =\frac{1}{8}(q-3+2s).
  \end{align*}
  If $q\equiv 5\ ({\rm{mod}}\ 8)$, then by Lemma \ref{20231012lemma1}, there exists an integer $s$ such that
  \begin{align*}
    (1,0)_4^{(q)}+(3,0)_4^{(q)} & =\frac{1}{16}(q-3-2s)+\frac{1}{16}(q-3-2s) \\
                                & =\frac{1}{8}(q-3-2s).
  \end{align*}
  In either case, we have
  $$(1,0)_4^{(q)}+(3,0)_4^{(q)}=\frac{1}{8}(q-3\pm 2s).$$
  \begin{enumerate}[label=(\arabic*)]
    \item If $p\equiv 1\ ({\rm{mod}}\ 4)$, then the above integers $s$ and $t$ satisfy that $q=s^2+4t^2$, $p\nmid st$ (in particular, $t\ne 0$) and $s\equiv 1\ ({\rm{mod}}\ 4)$. It follows that
          \begin{align*}
            (1,0)_4^{(q)}+(3,0)_4^{(0)}=0 & \iff q-3\pm 2s=0                                                         \\
                                          & \iff (s\pm 1)^2+4t^4=4                                                   \\
                                          & \iff s\pm 1=0,\ t\in\{\pm 1\}\quad(\mbox{since}\ t\ne 0)                 \\
                                          & \iff s=1,\ t\in\{\pm 1\} \quad(\mbox{since}\ s\equiv 1\ ({\rm{mod}}\ 4)) \\
                                          & \iff q=5,
          \end{align*}
          which contradicts the assumption that $q\ge 13$.
    \item If $p\equiv 3\ ({\rm{mod}}\ 4)$, then $s\in\{\pm p^{n/2}\}$. It follows that
          \begin{align*}
            (1,0)_4^{(q)}+(3,0)_4^{(0)}=0 & \iff q-3\pm 2s=0         \\
                                          & \iff (\sqrt{q}\pm 1)^2=4 \\
                                          & \iff q=9,
          \end{align*}
          which contradicts the assumption that $q\ge 13$.
  \end{enumerate}
  Hence in all cases, we have $(1,0)_4^{(q)}+(3,0)_4^{(q)}\ge 1$. This completes the proof.
\end{proof}

Theorem \ref{20231110them2} can be generalized to the case of general subset $S$ of $[m-1]$. The proof is similar, so we omit it.

\begin{theorem}\label{20231012them3}
  Keep the notations in Theorem \ref{20231110them1}. Then for any non-empty subset $S$ of $[m-1]$, the family $\{A_0,\cdots,A_{m-1}\}$ is a $(q,m,l;|S|)$-$S$-CEDF in $(\mathbb{F}_q,+)$ if and only if for any coset of $H=\left\langle\beta\right\rangle$ in $\mathbb{F}_q^*$, there are exactly $|S|$ elements in the following set
  $$\displaystyle\bigcup_{c\in S}\{x-1:\ x\in A_c\}=\{\beta^{c+km}-1:\ 0\le k\le l-1,\ c\in S\}$$
  that belong to this coset.
\end{theorem}

\begin{table}[htbp]
  \centering
  \caption{The remainders $2^i({\rm{mod}}\ 37)$ $(0\le i\le 18)$}\label{20231012table1}
  \begin{tabular}{ccccccccccc}
    \toprule
    $i$                   & $0$  & $1$  & $2$  & $3$  & $4$  & $5$  & $6$  & $7$  & $8$                \\
    \midrule
    $2^i({\rm{mod}}\ 37)$ & $1$  & $2$  & $4$  & $8$  & $16$ & $32$ & $27$ & $17$ & $34$               \\
    \midrule
    $i$                   & $9$  & $10$ & $11$ & $12$ & $13$ & $14$ & $15$ & $16$ & $17$ & $18$        \\
    \midrule
    $2^i({\rm{mod}}\ 37)$ & $31$ & $25$ & $13$ & $26$ & $15$ & $30$ & $23$ & $9$  & $18$ & $36\ (=-1)$ \\
    \bottomrule
  \end{tabular}
\end{table}
\begin{example}
  Take $q=37$, $m=4$ and $l=3$. Then $q=ml^2+1$. The remainders $2^i({\rm{mod}}\ 37)$ $(0\le i\le 18)$ are listed in Table \ref{20231012table1}. It is clear that $\theta=2$ is a primitive element in $\mathbb{F}_{q}$. Put $\beta=\theta^l=8$, $H=\left\langle\beta\right\rangle=\left\langle 8\right\rangle$, $C=\left\langle\beta^m\right\rangle=\left\langle 2^{12}\right\rangle=\{1,2^{12},2^{24}\}=\{1,26,10\}$ and $A_j=\beta^jC$ $(0\le j\le 3)$. Then
  \begin{align*}
    A_0=\{1,26,10\},\quad A_1=\{8,23,6\},\quad A_2=\{27,36,11\},\quad A_3=\{31,29,14\}.
  \end{align*}
  Let $\chi:\mathbb{F}_{q}^*\rightarrow\mathbb{C}$ be the cubic character of $\mathbb{F}_{q}^*$ which maps $2$ to $\omega=\exp(2\pi{\rm{i}}/3)$. Then for any $x\in\mathbb{F}_q^*$, we have
  $$\chi(x)=\begin{cases}
      1,        & x\in H=\{\pm 1,\pm 6,\pm8,\pm10,\pm11,\pm14\},  \\
      \omega,   & x\in 2H=\{\pm 2,\pm9,\pm12,\pm15,\pm16,\pm17\}, \\
      \omega^2, & x\in 4H=\{\pm 3,\pm4,\pm5,\pm7,\pm13,\pm18\}.
    \end{cases}$$
  By Theorem \ref{20231012them3}, for any non-empty subset $S\subset [m-1]=\{1,2,3\}$, the family $\{A_0,A_1,A_2,A_3\}$ is a $(37,4,l;|S|)$-$S$-CEDF if and only if for any coset of $H$ in $\mathbb{F}_q^*$, there are exactly $|S|$ elements in the following set
  $$\displaystyle\bigcup_{c\in S}\{x-1:\ x\in A_c\}=\{8^{3k+c}-1:\ 0\le k\le 2,\ c\in S\}$$
  that belong to this coset. The latter is equivalent to saying that the multiset
  $$\{\{\chi(x-1):\ x\in A_c,\ c\in S\}\}$$
  is the union of $|S|$ copies of $\{1,\omega,\omega^2\}$.
  Since
  \begin{align*}
     & \mathcal{A}_1=\{\{\chi(x-1)|\ x\in A_1\}\}=\{\{\chi(7),\chi(22),\chi(5)\}\}=\{\{\omega^2,\omega,\omega^2\}\},   \\
     & \mathcal{A}_2=\{\{\chi(x-1)|\ x\in A_2\}\}=\{\{\chi(26),\chi(35),\chi(10)\}\}=\{\{1,\omega,1\}\},               \\
     & \mathcal{A}_3=\{\{\chi(x-1)|\ x\in A_3\}\}=\{\{\chi(30),\chi(28),\chi(13)\}\}=\{\{\omega^2,\omega,\omega^2\}\},
  \end{align*}
  the family $\{A_0,A_1,A_2,A_3\}$ is not a $(37,4,l;1)$-$c$-CEDF for any $c\in\{1,2,3\}$. On the other hand, since
  $$\mathcal{A}_1\cup\mathcal{A}_2=\mathcal{A}_2\cup\mathcal{A}_3=\{\{1,1,\omega,\omega,\omega^2,\omega^2\}\},$$
  the family $\{A_0,A_1,A_2,A_3\}$ is a $(37,4,l;2)$-$|S|$-CEDF with $S=\{1,2\}$ or $\{2,3\}$. Example 4.7 in \cite{veitch2023unconditionally} is exactly the case where $S=\{1,2\}$.
\end{example}

In the above cyclotomic construction, the chosen sets $A_j=C_{jl}$ $(0\le j\le m-1)$ are exactly the cyclotomic classes of $\mathbb{F}_q$ of order $ml$ whose indices form the arithmetic sequence $0$, $l$, $2l$, $\cdots$, $(m-1)l$. If we replace them with any $m$ distinct cyclotomic classes of order $ml$, then the following theorem tells us when the chosen family is a CEDF.

\begin{theorem}
  Let $q$ be a prime power such that $q-1=ml^2$ with $l$, $m\ge 2$,  $\theta$ be a primitive element in $\mathbb{F}_q$,  $\beta=\theta^l$, $C=\left\langle\beta^m\right\rangle=\left\langle\theta^{ml}\right\rangle$ and  $C_i=\theta^iC$ $(0\le i\le ml-1)$ be the cyclotomic classes of $\mathbb{F}_q$ of order $ml$ $($with respect to $\theta)$.\\
  \indent Assume that $\sigma:\mathbb{Z}_m\rightarrow\mathbb{Z}_{ml}$ is an injective mapping. For any $0\le j\le m-1$, put $D_{j}=\theta^{\sigma(j)}C=C_{\sigma(j)}$. Then for any non-empty subset $S\subset [m-1]$, the family $\{D_0,\cdots,D_{m-1}\}$ is a $(q,m,l;|S|)$-$S$-CEDF in $(\mathbb{F}_q,+)$ if and only if for any coset of $C$ in $\mathbb{F}_q^*$, there are exactly $|S|$ elements in the following multiset
  $$\{\{\theta^{\sigma(j+c)}-\theta^{\sigma(j)}\beta^{mr}:\ 0\le j\le l-1,\ 0\le r\le m-1,\ c\in S\}\}$$
  that belong to this coset.
\end{theorem}
\begin{proof}
  Since $\sigma:\mathbb{Z}_m\rightarrow\mathbb{Z}_{ml}$ is injective, $D_j=\theta^{\sigma(j)}C$ $(0\le j\le m-1)$ are pairwise disjoint. We have
  \begin{align*}
    \displaystyle\bigcup_{c\in S}\displaystyle\bigcup_{j=0}^{m-1}\Delta(D_{j+c},D_j) & =\displaystyle\bigcup_{c\in S}\displaystyle\bigcup_{j=0}^{m-1}\Delta(C_{\sigma(j+c)},D_{\sigma(j)})                                         \\
                                                                                     & =\displaystyle\bigcup_{c\in S}\displaystyle\bigcup_{j=0}^{m-1}\{\theta^{\sigma(j+c)}x-\theta^{\sigma(j)}y:\ x,y\in C\}                      \\
                                                                                     & = \displaystyle\bigcup_{c\in S}\displaystyle\bigcup_{j=0}^{m-1}\{(\theta^{\sigma(j+c)}-\theta^{\sigma(j)}yx^{-1})x:\ x,y\in C\}             \\
                                                                                     & = \displaystyle\bigcup_{c\in S}\displaystyle\bigcup_{j=0}^{m-1}\{(\theta^{\sigma(j+c)}-\theta^{\sigma(j)}z)x:\ x,z\in C\}                   \\
                                                                                     & =\displaystyle\bigcup_{c\in S}\displaystyle\bigcup_{j=0}^{m-1}\{\theta^{\sigma(j+c)}-\theta^{\sigma(j)}z:\ z\in C\}\cdot C                  \\
                                                                                     & =\displaystyle\bigcup_{c\in S}\displaystyle\bigcup_{j=0}^{m-1}\{\theta^{\sigma(j+c)}-\theta^{\sigma(j)}\beta^{mr}:\ 0\le r\le l-1\}\cdot C.
  \end{align*}
  Hence the family $\{D_0,\cdots,D_{m-1}\}$ is a $(q,m,l;|S|)$-$S$-CEDF in $(\mathbb{F}_q,+)$, i.e.,
  $$\sum\limits_{c\in S}\sum\limits_{j=0}^{m-1}D_{j+c}D_j^{(-1)}=|S|\cdot\mathbb{F}_q^*,$$
  if and only if for any coset of $C$ in $\mathbb{F}_q^*$, there are exactly $|S|$ elements in the following multiset
  $$\{\{\theta^{\sigma(j+c)}-\theta^{\sigma(j)}\beta^{mr}:\ 0\le j\le m-1,\ 0\le r\le l-1,\ c\in S\}\}$$
  that belong to this coset.
\end{proof}

\begin{example}
  Take $q=19$, $m=2$ and $l=3$. Then $q=ml^2+1$. The remainders $2^i({\rm{mod}}\ 19)$ $(0\le i\le 9)$ are listed in Table \ref{20231012table2}. It is clear that $\theta=2$ is a primitive element in $\mathbb{F}_{q}$. Put $\beta=\theta^l=8$, $H=\left\langle\beta\right\rangle=\left\langle 8\right\rangle$, $C=\left\langle\beta^m\right\rangle=\left\langle 2^{6}\right\rangle=\{1,2^{6},2^{12}\}=\{1,7,11\}$, $C_i=\theta^iC$ $(0\le i\le 5)$ and $A_j=\beta^jC$ $(0\le j\le 1)$. Then
  \begin{align*}
    A_0=\{1,7,11\},\quad C_1=\{2,14,3\},\quad A_1=\{8,18,12\}.
  \end{align*}
  and
  \begin{align*}
    \Delta(A_0,A_1) & =\{x-y:\ x\in\{1,7,11\},\ y\in\{8,18,12\}\}  \\
                    & =\{2,3,8,8,12,12,14,18,18\},                 \\
    \Delta(A_1,A_0) & =-\Delta(A_0,A_1)=\{17,16,11,11,7,7,5,1,1\}.
  \end{align*}
  Since $\Delta(A_0,A_1)\cup\Delta(A_1,A_0)\ne\mathbb{F}_{19}^*$, the family $\{A_0,A_1\}$ is not a $(19,2,3;1)$-$1$-CEDF in $(\mathbb{F}_{19},+)$.\\
  \indent On the other hand, if we take $D_0=C=\{1,7,11\}$ and $D_1=C_1=\{2,14,3\}$, then
  \begin{align*}
    \Delta(D_0,D_1) & =\{x-y:\ x\in\{1,7,11\},\ y\in\{2,14,3\}\}    \\
                    & =\{4,5,6,8,9,12,16,17,18\},                   \\
    \Delta(D_1,D_0) & =-\Delta(D_0,D_1)=\{15,14,13,11,10,7,3,2,1\}.
  \end{align*}
  Since $\Delta(D_0,D_1)\cup\Delta(D_1,D_0)=\mathbb{F}_{19}^*$, the family $\{D_0,D_1\}$ is a $(19,2,3;1)$-$1$-CEDF in $(\mathbb{F}_{19},+)$.

  \begin{table}[htbp]
    \centering
    \caption{The remainders $2^i({\rm{mod}}\ 19)$ $(0\le i\le 9)$}\label{20231012table2}
    \begin{tabular}{ccccccccccc}
      \toprule
      $i$                   & $0$ & $1$ & $2$ & $3$ & $4$  & $5$  & $6$ & $7$  & $8$ & $9$         \\
      \midrule
      $2^i({\rm{mod}}\ 19)$ & $1$ & $2$ & $4$ & $8$ & $16$ & $13$ & $7$ & $14$ & $9$ & $18\ (=-1)$ \\
      \bottomrule
    \end{tabular}
  \end{table}
\end{example}

We conclude this section with a lifting theorem on CEDFs.

\begin{theorem}
  Let $q$ be a prime power such that $q-1=ml^2$ with $l$, $m\ge 2$, $\theta$ be a primitive element in $\mathbb{F}_q$, $\beta=\theta^l$,  $H=\left\langle\beta\right\rangle$,  $C=\left\langle\beta^m\right\rangle$ and $A_j=\beta^jC$ $(0\le j\le m-1)$.\\
  \indent Let $Q=q^t$ with $t\in\mathbb{N}_+$, $M=(Q-1)/l^2$, $\tilde{\theta}$ be a primitive element in $\mathbb{F}_Q$ such that $\theta=\tilde{\theta}^{(Q-1)/(q-1)}=\tilde{\theta}^{M/m}$, $\tilde{\beta}=\tilde{\theta}^l$, $\tilde{H}=\left\langle\tilde{\beta}\right\rangle$, $\tilde{C}=\left\langle\beta^M\right\rangle$ and $\tilde{A}_k=\tilde{\beta}^k\tilde{C}$ $(0\le k\le M-1)$.\\
  \indent Assume further that $\rm{gcd}(l,t)=1$. Then for any non-empty subset $S\subset[m-1]$, the family $\{A_0,\cdots,A_{m-1}\}$ is a $(q,m,l;|S|)$-$S$-CEDF in $(\mathbb{F}_q,+)$ if and only if the family $\{\tilde{A}_0,\cdots,\tilde{A}_{M-1}\}$ is a $(Q,M,l;|\tilde{S}|=|S|)$-$\tilde{S}$-CEDF in $(\mathbb{F}_Q,+)$, where $\tilde{S}=\frac{M}{m}S=\{\frac{M}{m}c:\ c\in S\}\subset [M-1]$.
\end{theorem}

\begin{proof}
  Put
  $$B=\{\beta^{c+km}-1:\ 0\le k\le l-1,\ c\in S\}$$
  and
  $$\tilde{B}=\{\tilde{\beta}^{c+kM}-1:\ 0\le k\le l-1,\ \tilde{c}\in\tilde{S}\}.$$
  By Theorem \ref{20231012them3}, the family $\{A_0,\cdots,A_{m-1}\}$ $($resp., $\{\tilde{A}_0,\cdots,\tilde{A}_{m-1}\})$ is a $(q,m,l;|S|)$-$S$-CEDF $($resp., $(q,m,l;|\tilde{S}|)$-$\tilde{S}$-CEDF) in $(\mathbb{F}_q,+)$ $($resp., $(\mathbb{F}_Q,+))$ if and only if for any coset of $H$ $($resp., $\tilde{H})$ in $\mathbb{F}_q^*$ $($resp., $\mathbb{F}_Q^*)$, there are exactly $|S|=|\tilde{S}|$ elements in the set $B$ $($resp., $\tilde{B})$ that belong to this coset. Since
  \begin{align*}
    B & =\{\beta^{c+km}-1:\ 0\le k\le l-1,\ c\in S\}                                           \\
      & =\{\tilde{\beta}^{\frac{M}{m}(c+km)}-1:\ 0\le k\le l-1,\ c\in S\}                      \\
      & =\{\tilde{\beta}^{\frac{M}{m}c+kM}-1:\ 0\le k\le l-1,\ c\in S\}                        \\
      & =\{\tilde{\beta}^{\tilde{c}+kM}-1:\ 0\le k\le l-1,\ \tilde{c}\in\tilde{S}\}=\tilde{B},
  \end{align*}
  it suffices to show that the natural embedding $\mathbb{F}_q^*\rightarrow\mathbb{F}_Q^*$ induces a group isomorphism $\mathbb{F}_q^*/H\stackrel{\cong}{\rightarrow}\mathbb{F}_Q^*/\tilde{H}$. Since both $\mathbb{F}_q^*/H$ and $\mathbb{F}_Q^*/\tilde{H}$ are a cyclic group of order $l$, and the image of $\theta$ in $\mathbb{F}_q^*/H$ is a generator of $\mathbb{F}_q^*/H$, it suffices to show that the order of $($the image of$)$ $\theta$ in $\mathbb{F}_Q^*/\tilde{H}$ is $l$. Since $\theta=\tilde{\theta}^{M/m}$ and the order of $\tilde{\theta}$ in $\mathbb{F}_Q^*/\tilde{H}$ is $l$, the order of $\theta$ in $\mathbb{F}_Q^*/\tilde{H}$ equals $l/\gcd(l,M/m)$. Hence it suffices to show that $\gcd(l,M/m)=1$. Indeed, since $l\mid (q-1)$, i.e., $q\equiv 1\ ({\rm{mod}}\ l)$, we have
  $$\gcd(l,\frac{M}{m})=\gcd(l,\frac{Q-1}{q-1})=\gcd(l,q^{t-1}+q^{t-2}+\cdots+q+1)=\gcd(l,t)=1.$$
  This completes the proof.
\end{proof}

\section{Strong Circular External Difference Families}

We start by recalling the definition of the strong circular external difference family.

\begin{definition}[{\cite[Definition 4.7]{veitch2023unconditionally}}]
  Let $(G,\cdot)$ be a finite abelian group of order $n$, let $l$, $m\ge 2$ and let $c\in [m-1]$. A family $\{A_0,\cdots,A_{m-1}\}$ of $m$ pairwise disjoint subsets of $G$ is called an $(n,m,l;\lambda)$-$c$-strong circular external difference family (SCEDF) in $G$ if
  \begin{enumerate}[label=(\arabic*)]
    \item $|A_i|=l$ for any $0\le i\le m-1$;
    \item for any $0\le i\le m-1$, each element $g\in G\backslash\{1_G\}$ occurs exactly $\lambda$ times in $\Delta(A_{i+c},A_i)$, or equivalently,
          $$A_{i+c}A_i^{(-1)}=\lambda(G-1_G),$$
          where the addition of the subscripts is performed in the sense of modulo $m$.
  \end{enumerate}
  Obviously, a necessary condition for the existence of an $(n,m,l;\lambda)$-$c$-CEDF is $l^2=\lambda(n-1)$.
\end{definition}

The concept of SCEDF was proposed in \cite{veitch2023unconditionally} by Veitch and Stinson in order to construct strong circular algebraic manipulation detection codes where the secret is not uniformly distributed. On page $14$ of \cite{veitch2023unconditionally}, they stated that in general, SCEDFs seem difficult to construct, so the first question they asked at the end of \cite{veitch2023unconditionally} was whether there exist any non-trivial examples of SCEDFs.

In this section, we will give a complete answer to this question by showing the non-existence of non-trivial SCEDFs. Before that, we should first clarify the meaning of triviality. Essentially, trivial SCEDFs are those which are constructed from $(n,2,l;\lambda$)-SEDFs. It is not difficult to see that an $(n,2,l;\lambda)$-$1$-SCEDF is the same thing as an $(n,2,l;\lambda)$-SEDF. The following two lemmas tell us how to construct larger SCEDFs from such SEDFs.

\begin{lemma}\label{20231013lemmaone}
  Let $(G,\cdot)$ be a finite abelian group of order $n$, let $c\in [m-1]$ and let $\{A_0,\cdots,A_{m-1}\}$ be a family of $m$ disjoint subsets of $G$ such that $|A_i|=l$ for any $0\le i\le m-1$. Put $t=\gcd(c,m)$, $c'=c/t$, $m'=m/t$, and
  $$\mathcal{D}^{(i)}=\{D_0^{(i)}=A_i,\ D_1^{(i)}=A_{i+t},\ \cdots,\ D_{m'-1}^{(i)}=A_{i+(m'-1)t}\}$$
  for any $0\le i\le t-1$. Then the family $\{A_0,\cdots,A_{m-1}\}$ is an $(n,m,l;\lambda)$-$c$-SCEDF in $G$ if and only if for any $0\le i\le t-1$, the family $\mathcal{D}^{(i)}$ is an $(n,m',l;\lambda)$-$c'$-SCEDF in $G$.
\end{lemma}
\begin{proof}
  Since any integer $r\in\{0,\cdots,m-1\}$ can be uniquely written as $r=i+jt$ with $0\le i\le t-1$ and $0\le j\le m'-1$, the family $\{A_0,\cdots,A_{m-1}\}$ is the disjoint union of the families $\mathcal{D}^{(0)},\cdots,\mathcal{D}^{(t-1)}$. Moreover, for any $0\le i\le t-1$ and any $0\le j\le m'-1$, we have $D_{j+m'}^{(i)}=A_{i+(j+m')t}=A_{i+jt+m}=A_{i+jt}=D_j^{(i)}$ and
  $$\Delta(D_{j+c'}^{(i)},D_j^{(i)})=\Delta(A_{i+(j+c')t},A_{i+jt})=\Delta(A_{(i+jt)+c},A_{i+jt}).$$
  This lemma follows immediately.
\end{proof}

\begin{lemma}\label{20231013lemmatwo}
  Let $(G,\cdot)$ be a finite abelian group of order $n$, let $c\in [m-1]$ and let $\{A_0,\cdots,A_{m-1}\}$ be a family of $m$ disjoint subsets of $G$ such that $|A_i|=l$ for any $0\le i\le m-1$. For any $0\le i\le m-1$, put $D_i=A_{ic}$. Then the family $\{A_0,\cdots,A_{m-1}\}$ is an $(n,m,l;\lambda)$-$c$-SCEDF in $G$ if and only if the family $\{D_0,D_1,\cdots,D_{m-1}\}$ is an $(n,m,l;\lambda)$-$1$-SCEDF in $G$.
\end{lemma}
\begin{proof}
  Since $\gcd(c,m)=1$, the family $\{D_0,\cdots,D_{m-1}\}$ is a permutation of the family $\{A_0,\cdots,A_{m-1}\}$. Moreover, for any $0\le i\le m-1$, we have $D_{i+m}=A_{(i+m)c}=A_{ic+mc}=A_{ic}=D_i$ and
  $$\Delta(D_{i+1},D_i)=\Delta(A_{(i+1)c},A_{ic})=\Delta(A_{ic+c},A_{ic}).$$
  This lemma follows immediately.
\end{proof}

Now assume that $\mathcal{D}^{(i)}=\{D_j^{(0)}, D_j^{(1)}\}$ are $c$ disjoint $(n,2,l;\lambda)$-SEDFs in a finite abelian group $(G,\cdot)$. For any integer $r\in\{0,\cdots,2c-1\}$, put $A_r=D^{(i)}_j$, where $i$, $j$ are the unique integers with $0\le i\le c-1$ and $0\le j\le 1$ such that $r=i+jc$. By Lemma \ref{20231013lemmaone}, the family $\{A_0,\cdots,A_{2c-1}\}$ is an $(n,2c,l;\lambda)$-$c$-SCEDF in $G$. An SCEDF constructed in this way will be called a trivial SCEDF. Reversing the process, one can easily see that any $(n,m,l;\lambda)$-$c$-SCEDF with $m=2c$ is trivial.

\begin{theorem}\label{20231013themone}
  Let $\{A_0,\cdots,A_{m-1}\}$ be an $(n,m,l;\lambda)$-$1$-SCEDF in a finite abelian group $(G,\cdot)$ of order $n$. Then $m=2$.
\end{theorem}

\begin{proof}
  By definition, the family $\{A_0,\cdots,A_{m-1}\}$ is an $(n,m,l;\lambda)$-$1$-SCEDF in $G$ if and only if
  \begin{equation}\label{20231013equationone}
    A_{i+1}A_i^{(-1)}=\lambda(G-1_G)\quad\mbox{in}\ \mathbb{Z}[G]
  \end{equation}
  for any $0\le i\le m-1$. It follows that for any non-trivial character $\chi\in\widehat{G}$ and any $0\le i\le m-1$, we have
  $$\chi(A_{i+1})\overline{\chi(A_i)}=\chi(A_{i+1})\chi(A_i^{(-1)})=\lambda\cdot\chi(G)-\lambda\cdot\chi(1_G)=-\lambda\quad\mbox{in }\mathbb{C}.$$
  In particular, we have
  $$\chi(A_{1})\overline{\chi(A_0)}=-\lambda=\chi(A_2)\overline{\chi(A_1)}.$$
  Since $\lambda\ne 0$, we have $\chi(A_1)\ne 0$. Since $\lambda\in\mathbb{R}$, we have
  $$\chi(A_0)\overline{\chi(A_1)}=\overline{\chi(A_1)\overline{\chi(A_0)}}=\overline{-\lambda}=-\lambda=\chi(A_2)\overline{\chi(A_1)}.$$
  Since $\chi(A_1)\ne 0$, we have $\chi(A_0)=\chi(A_2)$. Moreover, we have
  $$1_{\widehat{G}}(A_0)=|A_0|=l=|A_2|=1_{\widehat{G}}(A_2).$$
  Hence $\chi(A_0)=\chi(A_2)$ for any $\chi\in\widehat{G}$, which implies that $A_0=A_2$.
\end{proof}

Combining the results above, we can derive the non-existence of non-trivial SCEDFs.

\begin{corollary}
  There are no non-trivial SCEDFs.
\end{corollary}
\begin{proof}
  Assume that the family $\mathcal{A}=\{A_0,\cdots,A_{m-1}\}$ is an $(n,m,l;\lambda)$-$c$-SCEDF with $m\ne 2c$ in a finite abelian group $G$. By Lemma \ref{20231013lemmaone}, there exists an $(n,m',l;\lambda)$-$c'$-SCEDF in $G$ with $m'=m/t$, $c'=c/t$, where $t=\gcd(c,m)$. By Lemma \ref{20231013lemmatwo}, we may assume that $c'=1$. Since $c\in[m-1]$ and $m\ne 2c$, we must have $t=\gcd(c,m)<\frac{m}{2}$ and thus $m'=m/t>2$. By Theorem \ref{20231013themone}, this is impossible. Hence all SCEDFs are trivial.
\end{proof}

Therefore, the study of SCEDFs reduces to the study of $(n,2,l;\lambda)$-SEDFs. A comprehensive summary of all the constructed SCEDFs of this type can be found in \cite[Proposition 1.1]{jedwab2019construction}. On the other hand, we have the following result on the non-existence of such  SCEDFs.

\begin{theorem}[{\cite[Theorem 4.2]{jedwab2019construction}}]\label{20231013themthree}
  Let $(G,\cdot)$ be a finite abelian group of order $n$ and let $\lambda$, $l$ be positive integers such that $l\ge 2$ and $l^2=\lambda(n-1)$. Assume that
  \begin{enumerate}[label=(\arabic*)]
    \item $p$ is a prime factor of $n$ and $q$ is a prime factor of $\lambda$ such that $d=\max\{i\in\mathbb{N}_+:\ p^i\mid n\}$ and $f=\max\{j\in\mathbb{N}_+:\ q^j\mid\lambda\}$;
    \item $q$ is a primitive root modulo $p^d$, i.e., $q$ is a generator of $\mathbb{Z}_{p^d}^*$;
    \item $|G_p|>n/q^{\lceil f/2\rceil}$, where $G_p$ is the Sylow $p$-subgroup of $G$, i.e., the maximal cyclic subgroup of $G$ whose order is a power of $p$.
  \end{enumerate}
  Then there does not exist any $(n,2,l;\lambda)$-SEDF in $G$.
\end{theorem}

If $n=p$ in Theorem \ref{20231013themthree}, then $G=G_p$ and thus $|G_p|=p>n/q\ge n/q^{\lceil f/2\rceil}$, i.e., the condition (3) in Theorem \ref{20231013themthree} is automatically satisfied. Hence we have the following corollary.

\begin{corollary}\label{20231013corolone}
  Let $p$ be a prime number and let $\lambda$, $l$ be positive integers such that $2\le l\le p/2$ and $\lambda(p-1)=l^2$. If $\lambda$ has a prime factor $q$ which is a primitive root modulo $p$. Then there does not exist any $(p,2,l;\lambda)$-SEDF.
\end{corollary}
\begin{remark}
  Let $p$, $q$ be two prime numbers such that $p\ge 3$ and $q$ is a primitive root modulo $p$. We want to determine all integers $\lambda$, $l$ such that $2\le l\le(p-1)/2$, $q\mid\lambda$ and $\lambda(p-1)=l^2$. First, we decompose $p-1$ into the following form:
  $$p-1=q^cp_1^{2a_1}\cdots p_s^{2a_s}q_1^{2b_1-1}\cdots q_r^{2b_r-1},$$
  where $q$, $p_1$, $\cdots$, $p_s$, $q_1$, $\cdots$, $q_r$ are pairwise distinct prime numbers, $c\in\mathbb{N}$ and $a_1$, $\cdots$, $a_s$, $b_1$, $\cdots$, $b_r\in\mathbb{N}_+$. Since $p\ge 3$, $c+s+r\ge 1$. Put
  $$l_{\min}=q^{\lceil\frac{c+1}{2}\rceil}p_1^{a_1}\cdots p_s^{a_s}q_1^{b_1}\cdots q_r^{b_r}\ (\ge 2).$$
  If $\lambda(p-1)=l^2$ and $q\mid\lambda$, then $q(p-1)\mid l^2$, which implies that
  $l_{min}\mid l$. If $l_{\min}>\frac{p-1}{2}$, then there are no integers $\lambda$, $l$ satisfying the required conditions; otherwise, take $\lambda_{\min}=l_{\min}^2/(p-1)$ and then the integers $\lambda=\lambda_{\min}$, $l=l_{\min}$ satisfy the required conditions. More generally, for any integer $1\le t\le d:=\lfloor(p-1)/2l_{\min}\rfloor$, the integers $\lambda=t^2\lambda_{\min}$, $l=tl_{\min}$ satisfy the required conditions. By Corollary \ref{20231013corolone}, there do not exist any $(p,2,l;\lambda)$-SEDFs for such $\lambda$ and $l$. In Example \ref{20231014exampleone} below we give some specific instances.
\end{remark}

\begin{example}\label{20231014exampleone}
  If $p=4357$, then $q=2$ is a primitive root modulo $p$. Since $p-1=2^2\cdot 3^2\cdot 11^2$, we have $l_{\min}=2^2\cdot 3\cdot 11=132$, $\lambda_{\min}=l_{\min}^2/(p-1)=4$ and $d=\lfloor(p-1)/2l_{\min}\rfloor=16$. Hence for any integer $1\le t\le 16$, there does not exist any $(4357,2,132t;4t^2)$-SEDF. In Table \ref{20231012table4}, we list all cases where $p<1000$.
  \begin{table}[htbp]
    \centering
    \caption{Non-existent $(p,2,l;\lambda)$-SEDFs}\label{20231012table4}
    \begin{tabular}{cccccc}
      \toprule
      $p$   & $q$ & $p-1$                  & $l_{\min}$                & $\lambda_{\min}$    & $d=\lfloor\frac{p-1}{2l_{\min}}\rfloor$ \\
      \midrule
      $19$  & $2$ & $2\cdot 3^2$           & $2\cdot 3$                & $2$                 & $1$                                     \\
      $37$  & $2$ & $2^2\cdot 3^2$         & $2^2\cdot 3$              & $2^1$               & $1$                                     \\
      $101$ & $2$ & $2^2\cdot 5^2$         & $2^2\cdot 5$              & $2^2$               & $2$                                     \\
      $101$ & $3$ & $2^2\cdot 5^2$         & $2\cdot 5\cdot 3$         & $3^2$               & $1$                                     \\
      $163$ & $2$ & $2\cdot 3^4$           & $2\cdot 3^2$              & $2$                 & $4$                                     \\
      $163$ & $3$ & $2\cdot 3^4$           & $2\cdot 3^3$              & $2\cdot 3^2$        & $1$                                     \\
      $181$ & $2$ & $2^2\cdot 3^2\cdot 5$  & $2^2\cdot 3\cdot 5$       & $2^2\cdot 5$        & $1$                                     \\
      $197$ & $2$ & $2^2\cdot 7^2$         & $2^2\cdot 7$              & $2^2$               & $3$                                     \\
      $197$ & $3$ & $2^2\cdot 7^2$         & $2\cdot 7\cdot 3$         & $3^2$               & $2$                                     \\
      $197$ & $5$ & $2^2\cdot 7^2$         & $2\cdot 7\cdot 5$         & $5^2$               & $1$                                     \\
      $257$ & $3$ & $2^8$                  & $2^4\cdot 3$              & $3^2$               & $2$                                     \\
      $257$ & $5$ & $2^8$                  & $2^4\cdot 5$              & $5^2$               & $1$                                     \\
      $257$ & $7$ & $2^8$                  & $2^4\cdot 7$              & $7^2$               & $1$                                     \\
      $401$ & $3$ & $2^4\cdot 5^2$         & $2^2\cdot 3\cdot 5$       & $3^2$               & $3$                                     \\
      $433$ & $5$ & $2^4\cdot 3^3$         & $2^2\cdot 3^2\cdot 5$     & $3\cdot 5^2$        & $1$                                     \\
      $449$ & $3$ & $2^6\cdot 7$           & $2^3\cdot 3\cdot 7$       & $3^2\cdot 7$        & $1$                                     \\
      $487$ & $3$ & $2\cdot 3^5$           & $2\cdot 3^3$              & $2\cdot 3$          & $4$                                     \\
      $491$ & $2$ & $2\cdot 5\cdot 7^2$    & $2\cdot 5\cdot 7$         & $2\cdot 5$          & $3$                                     \\
      $541$ & $2$ & $2^2\cdot 3^3\cdot 5$  & $2^2\cdot 3^2\cdot 5$     & $2^2\cdot 3\cdot 5$ & $1$                                     \\
      $577$ & $5$ & $2^6\cdot 3^2$         & $2^3\cdot 3\cdot 5$       & $5^2$               & $2$                                     \\
      $577$ & $7$ & $2^6\cdot 3^2$         & $2^3\cdot 3\cdot 7$       & $7^2$               & $1$                                     \\
      $641$ & $3$ & $2^7\cdot 5$           & $2^4\cdot 3\cdot 5$       & $2\cdot 5\cdot 3^2$ & $1$                                     \\
      $677$ & $2$ & $2^2\cdot 13^2$        & $2^2\cdot 13$             & $2^2$               & $6$                                     \\
      $701$ & $2$ & $2^2\cdot 5^2\cdot 7$  & $2^2\cdot 5\cdot 7$       & $2^2\cdot 7$        & $2$                                     \\
      $727$ & $5$ & $2\cdot 3\cdot 11^2$   & $2\cdot 3\cdot 5\cdot 11$ & $2\cdot 3\cdot 5^2$ & $1$                                     \\
      $757$ & $2$ & $2^2\cdot 3^2\cdot 7$  & $2^2\cdot 3^2\cdot 7$     & $2^2\cdot 3\cdot 7$ & $1$                                     \\
      $811$ & $3$ & $2\cdot 3^4\cdot 5$    & $2\cdot 3^3\cdot 5$       & $2\cdot 3^2\cdot 5$ & $1$                                     \\
      $829$ & $2$ & $2^2\cdot 3^2\cdot 23$ & $2^2\cdot 3\cdot 23$      & $2^2\cdot 23$       & $1$                                     \\
      $883$ & $2$ & $2\cdot 3^2\cdot 7^2$  & $2\cdot 3\cdot 7$         & $2$                 & $10$                                    \\
      \bottomrule
    \end{tabular}
  \end{table}
\end{example}

We finish this section with a new result on the non-existence of $(n,2,l;\lambda)$-SEDFs.

\begin{theorem}\label{20231013themfour}
  Let $p$ be an odd prime number and let $\lambda$, $l$ be positive integers such that $2\le l\le p/2$ and $\lambda(2p-1)=l^2$. Assume that $\lambda$ has a prime factor $q$ which is a primitive root modulo $p$. Then there does not exist any $(2p,2,l;\lambda)$-SEDF.
\end{theorem}

\begin{proof}
  Suppose that the family $\{A_0,A_1\}$ is a $(2p,2,l;\lambda)$-SEDF in a finite abelian group $G$. Since $|G|=2p$ and $(2,p)=1$, we know that $G$ is a cyclic group. Let $g_0$ be a generator of $G$ and let $\chi$ be the character of $G$ of order $2p$ which maps $g_0$ to $\xi_{2p}=\exp(\pi{\rm{i}}/p)$. Put
  $$\alpha=\chi(A_0)=\sum\limits_{g\in A_0}\chi(g),\quad\beta=\chi(A_1^{(-1)})=\sum\limits_{g\in A_1}\overline{\chi(g)}=\overline{\chi(A_1)}.$$
  Since $\{A_0,A_1\}$ is a $(2p,2,l;\lambda)$-SEDF, we have $A_0A_1^{(-1)}=\lambda(G-1_G)$ in $\mathbb{Z}[G]$, which implies that
  $$\alpha\beta=\lambda\big(\chi(G)-\chi(1_G)\big)=-\lambda\quad\mbox{in\ }\mathbb{C}.$$
  Since $\lambda\ne 0$, we have $\alpha$, $\beta\ne 0$. Note that $\alpha$, $\beta$ lie in $\mathcal{O}_K=\mathbb{Z}[\xi_{2p}]=\mathbb{Z}[\xi_p]$ (since $\xi_{2p}=\xi_{2p}^{2p+1}=-\xi_{2p}^{p+1}=-\xi_p^{(p+1)/2}$), i.e., the ring of algebraic integers of the cyclotomic field $K=\mathbb{Q}(\xi_{2p})=\mathbb{Q}(\xi_p)$ (see \cite[Theorem 2.6]{washington1997introduction}), where $\xi_p=\exp(\pi{\rm{i}}/p)$.\\
  \indent By assumption, we have $q\mid\lambda$ in $\mathbb{Z}$, which implies that $q\mid\alpha\beta$ in $\mathcal{O}_K$. Since $q$ is a primitive root modulo $p$, $q\mathcal{O}_K$ is a prime ideal of $\mathcal{O}_K$, which implies that $q\mid\alpha$ or $q\mid\beta$ in $\mathcal{O}_K$. Without loss of generality, we may assume that $q\mid\alpha$ in $\mathcal{O}_K$.\\
  \indent For any $0\le i\le 2p-1$, put
  $$n_i=\begin{cases}
      1, & \mbox{if\ }g_0^i\in A_0,    \\
      0, & \mbox{if\ }g_0^i\notin A_0.
    \end{cases}$$
  Then $\sum\limits_{i=0}^{2p-1}n_i=|A_0|=l$. Moreover, we have
  \begin{align*}
    \alpha & =\sum\limits_{g\in A_0}\chi(g)  =\sum\limits_{i=0}^{2p-1}n_i\chi(g_0^i)=\sum\limits_{i=0}^{2p-1}n_i\xi_{2p}^i \\
           & =\sum\limits_{i=0}^{p-1}(n_i\xi_{2p}^i+n_{i+p}\xi_{2p}^{i+p})                                                 \\
           & =\sum\limits_{i=0}^{p-1}(n_i-n_{i+p})\xi_{2p}^i=\sum\limits_{i=0}^{p-1}a_i\xi_{2p}^i,
  \end{align*}
  where $a_i=n_i-n_{i+p}\in\{0,\pm 1\}$ for any $0\le i\le p-1$. Since $\alpha\ne 0$, at least one $a_i$ ($0\le i\le p-1$) is non-zero. If $a_i=n_i-n_{i+p}\ne 0$ for any $0\le i\le p-1$, then $n_i+n_{i+p}=1$ for any $0\le i\le p-1$, which implies that
  $$l=\sum\limits_{i=0}^{2p-1}n_i=\sum\limits_{i=0}^{p-1}(n_i+n_{i+p})=p.$$
  If $l=p$, then the equality $\lambda(n-1)=l^2$ becomes $\lambda(2p-1)=p^2$. Since $(2p-1,p)=1$ and $2p-1>1$, this is impossible. Hence there must exist $0\le i_0\le p-1$ such that $a_{i_0}=0$. Since $[K:\mathbb{Q}]=[\mathbb{Q}(\xi_p):\mathbb{Q}]=p-1$, the set
  $$\{\xi_{2p}^i:\ 0\le i\le p-1,\ i\ne i_0\}$$
  is a $\mathbb{Q}$-basis of $K$. Since
  $$q\mid\alpha\ \mbox{in }\mathcal{O}_K\quad\mbox{and}\quad\alpha=\sum\limits_{g\in A_0}\chi(g)=\sum\limits_{\substack{0\le i\le p-1\\i\ne i_0}}a_i\xi_{2p}^{i},$$
  we obtain that $q\mid a_i$ in $\mathbb{Z}$ for any $0\le i\le p-1$. Since $a_i\in\{0,\pm 1\}$ and at least one $a_i$ is non-zero, we must have $q=1$, which contradicts the assumption that $q$ is a prime. Hence $q\nmid\alpha$ in $\mathcal{O}_K$ and thus there does not exist any $(2p,2,l;\lambda)$-SEDF.
\end{proof}

\begin{remark}
  If the conditions of Theorem \ref{20231013themfour} hold and furthermore $q=2$, $2^3\nmid\lambda$, then the condition $|G_p|>n/q^{\lceil f/2\rceil}$ in Theorem \ref{20231013themthree} does not hold. Indeed, in this case, we have $|G_p|=p$, $1\le f\le 2$ and thus $n/q^{\lceil f/2\rceil}=2p/2=p=|G_p|$. Therefore, by Theorem \ref{20231013themfour}, the $(2p,2,l;\lambda)$-SEDFs in Table \ref{20231012table3} do not exist, which, however, cannot be derived from Theorem \ref{20231013themthree}.
  \begin{table}[htbp]
    \centering
    \caption{Non-existent $(2p,2,l;\lambda)$-SEDFs}\label{20231012table3}
    \begin{tabular}{ccccccc}
      \toprule
      $n=2p$ & $p$    & $q$ & $n-1$  & $l$         & $\lambda$   & $(2l\le n)$      \\
      \midrule
      $1226$ & $613$  & $2$ & $35^2$ & $70(2t+1)$  & $4(2t+1)^2$ & $(0\le t\le 3)$  \\
      $2602$ & $1301$ & $2$ & $51^2$ & $102(2t+1)$ & $4(2t+1)^2$ & $(0\le t\le 5)$  \\
      $7226$ & $3613$ & $2$ & $85^2$ & $170(2t+1)$ & $4(2t+1)^2$ & $(0\le t\le 10)$ \\
      \bottomrule
    \end{tabular}
  \end{table}
\end{remark}


\section{Conclusion}
In this paper, in order to address the questions raised in \cite{veitch2023unconditionally}, and also for the development of the theory, we have made an intensive study of CEDFs and SCEDFs.

For CEDFs, we have generalized the construction introduced in \cite{veitch2023unconditionally}, thereby obtaining many concrete examples of CEDFs. In particular, we have been able to construct an infinite class of CEDFs with $\lambda=1$ by showing the existence of a $(q,m,2;1)$-$c$-CEDF in the additive group of the finite field $\mathbb{F}_q$ for some $c\in [m-1]$, provided that $q=4m+1\ge 13$.

For SCEDFs, we have shown that all SCEDFs are trivial, so that the study of SCEDFs reduces to the study of $(n,2,l;\lambda)$-SEDFs. We have also presented new results on the non-existence of $(2p,2,l;\lambda)$-SEDFs, where $p$ is a prime number.

Note that the CEDFs we have constructed are all in the additive groups of finite fields, and our results on the non-existence of SEDFs can only deal with the case of certain cyclic groups. The construction and non-existence of CEDFs and SEDFs in more general abelian groups is a worthwhile topic.

\bibliography{sn-bibliography}% common bib file


\begin{thebibliography}{8}
\ifx \bisbn   \undefined \def \bisbn  #1{ISBN #1}\fi
\ifx \binits  \undefined \def \binits#1{#1}\fi
\ifx \bauthor  \undefined \def \bauthor#1{#1}\fi
\ifx \batitle  \undefined \def \batitle#1{#1}\fi
\ifx \bjtitle  \undefined \def \bjtitle#1{#1}\fi
\ifx \bvolume  \undefined \def \bvolume#1{\textbf{#1}}\fi
\ifx \byear  \undefined \def \byear#1{#1}\fi
\ifx \bissue  \undefined \def \bissue#1{#1}\fi
\ifx \bfpage  \undefined \def \bfpage#1{#1}\fi
\ifx \blpage  \undefined \def \blpage #1{#1}\fi
\ifx \burl  \undefined \def \burl#1{\textsf{#1}}\fi
\ifx \doiurl  \undefined \def \doiurl#1{\url{https://doi.org/#1}}\fi
\ifx \betal  \undefined \def \betal{\textit{et al.}}\fi
\ifx \binstitute  \undefined \def \binstitute#1{#1}\fi
\ifx \binstitutionaled  \undefined \def \binstitutionaled#1{#1}\fi
\ifx \bctitle  \undefined \def \bctitle#1{#1}\fi
\ifx \beditor  \undefined \def \beditor#1{#1}\fi
\ifx \bpublisher  \undefined \def \bpublisher#1{#1}\fi
\ifx \bbtitle  \undefined \def \bbtitle#1{#1}\fi
\ifx \bedition  \undefined \def \bedition#1{#1}\fi
\ifx \bseriesno  \undefined \def \bseriesno#1{#1}\fi
\ifx \blocation  \undefined \def \blocation#1{#1}\fi
\ifx \bsertitle  \undefined \def \bsertitle#1{#1}\fi
\ifx \bsnm \undefined \def \bsnm#1{#1}\fi
\ifx \bsuffix \undefined \def \bsuffix#1{#1}\fi
\ifx \bparticle \undefined \def \bparticle#1{#1}\fi
\ifx \barticle \undefined \def \barticle#1{#1}\fi
\bibcommenthead
\ifx \bconfdate \undefined \def \bconfdate #1{#1}\fi
\ifx \botherref \undefined \def \botherref #1{#1}\fi
\ifx \url \undefined \def \url#1{\textsf{#1}}\fi
\ifx \bchapter \undefined \def \bchapter#1{#1}\fi
\ifx \bbook \undefined \def \bbook#1{#1}\fi
\ifx \bcomment \undefined \def \bcomment#1{#1}\fi
\ifx \oauthor \undefined \def \oauthor#1{#1}\fi
\ifx \citeauthoryear \undefined \def \citeauthoryear#1{#1}\fi
\ifx \endbibitem  \undefined \def \endbibitem {}\fi
\ifx \bconflocation  \undefined \def \bconflocation#1{#1}\fi
\ifx \arxivurl  \undefined \def \arxivurl#1{\textsf{#1}}\fi
\csname PreBibitemsHook\endcsname

%%% 1
\bibitem[\protect\citeauthoryear{Storer}{1967}]{storer1967cyclotomy}
\begin{botherref}
\oauthor{\bsnm{Storer}, \binits{T.}}:
Cyclotomy and difference sets.
Lectures in Advanced Mathematics. Chicago: Markham Pub. Co
(1967)
\end{botherref}
\endbibitem

%%% 2
\bibitem[\protect\citeauthoryear{Levenshtein}{1971}]{levenshtein1971one}
\begin{barticle}
\bauthor{\bsnm{Levenshtein}, \binits{V.I.}}:
\batitle{One method of constructing quasilinear codes providing synchronization in the presence of errors}.
\bjtitle{Problemy Peredachi Informatsii}
\bvolume{7}(\bissue{3}),
\bfpage{30}--\blpage{40}
(\byear{1971})
\end{barticle}
\endbibitem

%%% 3
\bibitem[\protect\citeauthoryear{Ogata et~al.}{2004}]{ogata2004new}
\begin{barticle}
\bauthor{\bsnm{Ogata}, \binits{W.}},
\bauthor{\bsnm{Kurosawa}, \binits{K.}},
\bauthor{\bsnm{Stinson}, \binits{D.R.}},
\bauthor{\bsnm{Saido}, \binits{H.}}:
\batitle{New combinatorial designs and their applications to authentication codes and secret sharing schemes}.
\bjtitle{Discrete Mathematics}
\bvolume{279}(\bissue{1-3}),
\bfpage{383}--\blpage{405}
(\byear{2004})
\end{barticle}
\endbibitem

%%% 4
\bibitem[\protect\citeauthoryear{Huczynska and Johnson}{2023}]{huczynska2023internal}
\begin{barticle}
\bauthor{\bsnm{Huczynska}, \binits{S.}},
\bauthor{\bsnm{Johnson}, \binits{L.M.}}:
\batitle{Internal and external partial difference families and cyclotomy}.
\bjtitle{Discrete Mathematics}
\bvolume{346}(\bissue{3}),
\bfpage{113295}
(\byear{2023})
\end{barticle}
\endbibitem

%%% 5
\bibitem[\protect\citeauthoryear{Veitch and Stinson}{2023}]{veitch2023unconditionally}
\begin{botherref}
\oauthor{\bsnm{Veitch}, \binits{S.}},
\oauthor{\bsnm{Stinson}, \binits{D.R.}}:
Unconditionally secure non-malleable secret sharing and circular external difference families.
To appear in Designs, Codes and Cryptography. https://arxiv.org/abs/2305.09405
(2023)
\end{botherref}
\endbibitem

%%% 6
\bibitem[\protect\citeauthoryear{Paterson and Stinson}{2023}]{paterson2023circular}
\begin{botherref}
\oauthor{\bsnm{Paterson}, \binits{M.B.}},
\oauthor{\bsnm{Stinson}, \binits{D.R.}}:
Circular external difference families, graceful labellings and cyclotomy.
arXiv preprint arXiv:2310.02810
(2023)
\end{botherref}
\endbibitem

%%% 7
\bibitem[\protect\citeauthoryear{Jedwab and Li}{2019}]{jedwab2019construction}
\begin{barticle}
\bauthor{\bsnm{Jedwab}, \binits{J.}},
\bauthor{\bsnm{Li}, \binits{S.}}:
\batitle{Construction and nonexistence of strong external difference families}.
\bjtitle{Journal of Algebraic Combinatorics}
\bvolume{49},
\bfpage{21}--\blpage{48}
(\byear{2019})
\end{barticle}
\endbibitem

%%% 8
\bibitem[\protect\citeauthoryear{Washington}{1997}]{washington1997introduction}
\begin{bbook}
\bauthor{\bsnm{Washington}, \binits{L.C.}}:
\bbtitle{Introduction to Cyclotomic Fields}
vol. \bseriesno{83}.
\bpublisher{Springer},
\blocation{New {Y}ork}
(\byear{1997})
\end{bbook}
\endbibitem

\end{thebibliography}

\end{document}